\documentclass[12pt]{amsart}

\numberwithin{equation}{section}

\usepackage{amsmath,amsthm,amsfonts,eucal,}
\usepackage{xypic}
\usepackage{graphicx}
\usepackage{amssymb}
\usepackage{xcolor}
\usepackage{url}

\def\ca{{\mathcal A}}
\def\cb{{\mathcal B}}

\def\ch{{\mathcal H}}

\def\cu{{\mathcal U}}

\def\ga{{\mathfrak A}}

\def\bn{{\mathbb N}}

\def\a{\alpha}

\def\p{\pi}

\def\s{\sigma} 

\def\f{\varphi}  
  
\def\om{\omega}

\newtheorem{thm}{Theorem}[section]
\newtheorem{example}[thm]{Example}
\newtheorem{lem}[thm]{Lemma}

\newtheorem{prop}[thm]{Proposition}

\theoremstyle{definition}

\begin{document}
\title[Quasi-invariant states]
{Quasi-invariant states for the action of compact groups}
\author[Maria Elena Griseta]{Maria Elena Griseta}
\address{Maria Elena Griseta\\
Dipartimento di Matematica\\
Universit\`{a} degli studi di Bari\\
Via E. Orabona, 4, 70125 Bari, Italy}
\email{\texttt{mariaelena.griseta@uniba.it}}

\date{\today}
\begin{abstract}
We analyze a natural $C^*$-algebraic definition of $G$-quasi-invariant states for the automorphic action of a compact group $G$. We prove that, given a $G$-quasi-invariant state with central support, when the action of the group $G$ commutes with the modular group, its GNS representation is equivalent to that of a $G$-invariant state.   

\vskip0.1cm\noindent \\
{\bf Mathematics Subject Classification}: 46L05, 46L30, 37N20.  \\
{\bf Key words}: $C^*$-dynamical systems, quasi-invariant states, invariant states, modular group. 
\end{abstract}

\maketitle
\section{Introduction}
 States invariant under the action of a group $G$ through $*$-automorphism on a $C^*$-algebra $\ga$ are very often considered in various fields of mathematics and physics \cite{BR1, Sakai}.
A larger class of states is associated with quasi-invariant measures, which are an important topic in ergodic theory and for example the natural setting for the construction of type III factors \cite{Wal,T2}. More in details, if $(X,\cb)$ is a measurable space and $G$ 
denotes a group of automorphisms acting on it by
$$
x\in X  \mapsto \a_g(x)\in X\,,\quad g\in G\,,
$$
a measure  $\mu$ on  $(X,\cb)$ is said to be $G$-quasi-invariant if for any  $g \in G$ the transformed measure  $\mu \circ \a_g$ is equivalent to the measure  $\mu$, that is, these measures are absolutely continuous with respect to each other. \\
Recall also that, given $\mu$ and $\nu$ two regular Borel measures on the compact Hausdorff space $X$, the multiplication representations associated to them are equivalent if and only if $\mu$ and $\nu$ are absolutely continuous with respect to each other. Therefore, in this paper, given $G$ a compact group, by $G$-quasi-invariant state we shall mean a state $\om$ on $\ga$ such that, for every $g\in G$, one has
that $\p_{\om\circ\a_g}$ and $\pi_\om$ are unitarily equivalent representations.\\
Our main goal is to prove that, under suitable assumptions, any quasi-invariant state is "absolutely continuous" with respect to an invariant state.\\
We should note that there does not seem to be a generally agreed-upon definition of what a quasi-invariant state should be in noncommutative settings. For example, in \cite{HKK} quasi-invariance was defined in terms of quasi-equivalence of the representations $\p_{\om\circ\a_g}$ and $\pi_\om$.\\
Recently, a different approach was given in \cite{AD}. Here, given $G$ a group of $*$-automorphisms of a $*$-algebra $\ca$, and assuming that for every $g\in G$ there exists the Radon-Nikodym derivative $x_g\in\ca$,  the authors define a state $\f$ on $\ca$ to be $G$-quasi-invariant if 
\begin{equation}
\label{eq:AccDha}
\f(g(a))=\f(x_ga)\,, \quad a\in \ca\,. 
\end{equation}
In the paper mentioned above, the authors assume the existence of the Radon-Nikodym derivative. In these notes, however, we deduce the existence of this positive operator after proving some results.  \\ 
Successively, in \cite{DKY}, several properties of $G$-quasi-invariant states defined in \eqref{eq:AccDha} were studied in depth. Finally, Dharhi and Ricard, in \cite{DR}, characterize $G$-quasi-invariant states with uniformly bounded cocycle $x_g$, in terms of invariant normal faithful states on the $C^*$-algebra $\ca$.\\ 
In these note, we show that, given a $G$-quasi-invariant state $\om$ with central support, when the action of the group $G$ commutes with the modular group, its GNS representation is equivalent to that associated with the composition of $\om$ with the canonical conditional expectation, obtained by averaging the action of the group. To reach this result, we first prove some technical lemmas.   

\section{$G$-quasi-invariant states}

We consider the action of a compact group $G$ through automorphisms of a $C^*$-algebra
$\ga$. In other words, we are given a group homomorphism
$\a: G\rightarrow\textrm{Aut}(\ga)$, \emph{i.e.} $\a_{gh}=\a_g\circ \a_h$ for all $g, h\in G$.
We also need to make some assumptions on the continuity of $\a$. As usually done in the literature, we will work under the hypothesis
that $\a$ is pointwise norm continuous, namely that for every $a\in\ga$  the vector-valued function
$ G\ni g\mapsto \a_g(a)\in\ga$ is continuous, where $\ga$ is thought of as being given with its norm topology.

We recall that a  state $\om$ on $\ga$ is said $G$-invariant if for every $g\in G$ one has
$\om\circ\a_g=\om$.

Throughout this paper by $G$-quasi-invariant state we shall mean a state $\om$ on $\ga$ such that, for every $g\in G$, one has
that $\p_{\om\circ\a_g}$ and $\pi_	\om$ are unitarily equivalent representation.

As is known, the hypotheses we are working under allow us to define a (faithful) condition expectation, $E_G$, from
$\ga$ onto $\ga^G:=\{a\in\ga: \a_g(a)=a\,\,\textrm{for all}\, g\in G\}$. This is obtained by averaging the action of the group, that is
$$E_G(a):=\int_G \a_g(a)\textrm{d}g\, , a\in\ga\, .$$

Let us $\om$ be a $G$-quasi-invariant state. We will always assume that $\om$ has central support.
As is well known, this amounts to asking that the GNS vector $\xi_\om$ is separating for
the von Neumann algebra $\mathcal{R}_\om:=\pi_\om(\ga)''$. Therefore, the von Neumann algebra
$\pi_\om(\ga)''\subset \cb(\ch_\om)$ is in standard form, and we can consider the modular flow
associated with the cyclic and separating vector $\xi_\om$, which we will denote by
$\{\s_\om^t: t\in\mathbb{R}\}$.

\begin{lem}
\label{lem:impl}
Let $(G, \ga, \alpha)$ be a $C^*$-dynamical system, where $G$ is a compact group and
$\ga$ a $C^*$-algebra. For any state $\om$ on $\ga$, consider the following
statements:
\begin{itemize}
\item[(i)] $\pi_{\om\circ E_G}\cong \pi_\om$
\item[(ii)] For every $g\in G$, $\pi_\om\circ \a_g\cong \pi_\om$
\item[(iii)] $\om$ is a G-quasi-invariant state
\end{itemize}
Then $(i)\Rightarrow (ii)\Rightarrow (iii)$.
\end{lem}

\begin{proof}
For the implication  $(i)\Rightarrow (ii)$  note that from  $\pi_{\om\circ E_G}\cong \pi_\om$ one finds
$\pi_{\om\circ E_G}\circ \a_g\cong \pi_\om\circ \a_g$ for all $g\in G$. Since
$\pi_{\om\circ E_G}\circ \a_g\cong \pi_{\om\circ E_G\circ \a_g}=\pi_{\om\circ E_G}$ and $\pi_\om\circ \a_g\cong
\pi_{\om\circ \a_g}$, we have $\pi_{\om\circ E_G}\cong \pi_{\om\circ \a_g}$ for all $g\in G$. The conclusion now follows by transitivity.\\
The implication  $(ii)\Rightarrow (iii)$ is obvious.

\end{proof}

\begin{lem}
\label{lem:extension}
Let $(G, \ga, \alpha)$ be a $C^*$-dynamical system, where $G$ is a compact group and
$\ga$ a $C^*$-algebra. If $\om$ is a $G$-quasi-invariant state, then the action of $G$
can be lifted to  $\pi_\om(\ga)''$ through automorphisms $\{\widetilde{\a_g}: g\in G\}\subset {\rm Aut}(\pi_\om(\ga)'')$ such that 
$$\widetilde{\a_g}(\pi_\om(a))=\pi_\om(\a_g(a))\,,\,\,\,a\in \ga.$$
\end{lem}

\begin{proof}
It is a straightforward consequence of the existence of the interpolating isomorphism between 
$\pi_\om(\ga)''$ and $\pi_{\om\circ\a_g}(\ga)''$
and the fact that  $\pi_{\om\circ\a_g}\cong \pi_\om\circ\alpha_g$ for
every $g\in G$.
\end{proof}

\begin{lem}
\label{lem:exp}
Under the same hypotheses as the previous lemma and assuming that $\ga$ is separable, for any
$T\in\mathcal{R}_\om$ and $x\in\ch_\om$ the vector-valued function
$G\ni g\mapsto \widetilde{\a_g}(T)x\in \ch_\om$ is Bochner
integrable. For $T\in\mathcal{R}_\om$, define
$$\widetilde {E_G}[T]x:=\int_G \widetilde{\a_g}(T)x\,{\rm d}g\,, \,\, x\in\ch_\om\,, $$
then $\widetilde{E_G}[T]$ lies in $\mathcal{R}_\om$. \\
Furthermore, the map
$\widetilde{E_G}:\mathcal{R}_\om\rightarrow\mathcal{R}_\om$ is
faithful and normal: $T\in\mathcal{R}_\om$ with $\widetilde {E_G}[T^*T]=0$ implies $T=0$, and
$\varphi\circ\widetilde {E_G}$ is a normal state for every normal state $\varphi$ on $\mathcal{R}_\om$.
\end{lem}

\begin{proof}
We start by noting that $\ch_\om$ is a separable Hilbert space. In particular, if 
$T\in\mathcal{R}_\om$, there exist  sequence $\{a_n: n\in\mathbb{N}\}\subset\ga$ such that 
$\pi_\om(a_n)$ strongly converges to $T$, and $\|\pi_\om(a_n)\|\leq M$ for every $n\in\mathbb{N}$. We next define a sequence $\{f_n: n\in\mathbb{N}\}$ of vector-valued functions on $G$
by $f_n(g):=\widetilde{\a_g}(\pi_\om(a_n))x$. Note that each $f_n$ is continuous (when $\ch_\om$ is endowed with the norm topology).
Further, if we set $f(g):= \widetilde{\a_g}(T)x\in \ch_\om$, we have that $f$ is the pointwise norm limit of the sequence $f_n$, and thus
$f$ is measurable.  At this point, to ascertain that $f$ is Bochner-integrable, it is enough to check
$\int_G \|f(g)\| \textrm{d}g<\infty$, which is obvious as $ \|f(g)\|\leq \|T\|\|x\|$ for all $g\in G$.\\
That $\widetilde {E_G}[T]$ sits in $\mathcal{R}_\om$ is a
consequence of von Neumann’s bicommutant theorem. In fact, if $S$ is in $\mathcal{R}_\om’$, for every
$x\in\ch_\om$ we have
\begin{align*}
S\widetilde {E_G}[T]x&=S\int_G \widetilde{\a_g}(T)x\,{\rm d}g=\int_G S \widetilde{\a_g}(T)x\,{\rm d}g
=\int_G  \widetilde{\a_g}(T)Sx\,{\rm d}g\\
&=\widetilde {E_G}[T]Sx\,,
\end{align*}
which shows that $\widetilde {E_G}[T]$ belongs to the double commutant of $\mathcal{R}_\om$, which coincides with
$\mathcal{R}_\om$.\\
Let us now prove that $\widetilde {E_G}$ is faithful.  Let $T\in\mathcal{R}_\om$ with $\widetilde {E_G}[T^*T]=0$.
For every $x\in \ch_\om$, we have $\langle \widetilde {E_G}[T^*T]x, x \rangle=0$, and thus
$\int_G \| \widetilde{\a_g}(T) x\|^2{\rm d}g=0$. Let $\{x_k: k\in\bn\}$ be a dense sequence in
$\ch_\om$ and let $N_k$ be a negligible set such that $\widetilde{\a_g}(T) x_k=0$ if $g\in G\setminus N_k$.
Set $N=\bigcup_{k\in\bn} N_k$. Since $N$ is negligible, $G\setminus N$ is non-empty. Now
if $g\in G\setminus N$, we have $\widetilde{\a_g}(T) x_k=0$ for every $k\in\bn$, that is $\widetilde{\a_g}(T)=0$, and so
$T=0$.\\
Finally, we have to show that
$\varphi\circ\widetilde{E_G}$ is a normal functional if $\varphi$ is normal.
To this aim, let
$\{T_n: n\in\bn\}\subset\mathcal{R}_\om$ be a monotone increasing sequence of positive operators and let $T\in\mathcal{R}_\om$ be
$\sup_n T_n$. 
There is no lack of generality if we assume
$\varphi(A)=\sum_{i=1}^\infty \langle A x_i, x_i \rangle$, for every $A\in\mathcal{R}_\om$,  with $\{x_i: i\in\bn\}$ being a sequence in
$\ch_\om$ such that $\sum_{i=1}^\infty \|x_i\|^2<\infty$. But then we have:
\begin{align*}
\lim_n \varphi(\widetilde E_G[T_n])&=\lim_n\sum_{i=1}^\infty \int_G \langle \widetilde{\a_g}(T_n) x_i, x_i \rangle\,{\rm d}g\\
&=\lim_n \int_G\sum_{i=1}^\infty\langle \widetilde{\a_g}(T_n) x_i, x_i \rangle\,{\rm d}g=\int_G\sum_{i=1}^\infty\langle \widetilde{\a_g}(T) x_i, x_i \rangle\,{\rm d}g\\
&=\varphi(\widetilde E_G[T])\,,
\end{align*}
thanks to the monotone convergence theorem.
\end{proof}

\begin{thm}
\label{thm:main}
Let $(G, \ga, \alpha)$ be a $C^*$-dynamical system, where $G$ is a compact group and
$\ga$ a separable $C^*$-algebra. If $\om$ is a $G$-quasi-invariant state  on  $\ga$ with central support
such that $\widetilde{\a_g}\circ\s_\om^t=\s_\om^t\circ\widetilde{\a_g}$ for all $g\in G$
and $t\in\mathbb{R}$, then one has 
$$\pi_{\om\circ E_G}\cong \pi_\om\ .$$\\
In particular, conditions (i), (ii) and (iii) in Lemma \ref{lem:impl} are equivalent.
\end{thm} 

\begin{proof}
Denote by $\widetilde{\om}$ the unique extension of $\om$ to $\mathcal{R}_\om$, that is
 $\widetilde{\om}(T)=\langle T\xi_\om, \xi_\om \rangle$, $T\in\mathcal{R}_\om$.
Let $\varphi_\om$ be the state on $\mathcal{R}_\om$ obtained as the composition
$\varphi_\om:=\widetilde{\om}\circ\widetilde{E_G}$, where
$\widetilde{E_G}$ is the map obtained above. Note that $\varphi_\om$ is a faithful normal state on
$\mathcal{R}_\om$ thanks to   Lemma \ref{lem:exp}.\\
We next show that $\varphi_\om\circ\s_\om^t=\varphi_\om$ for all  $t$. Since $\widetilde{\om}$ is clearly 
invariant under the action of its modular flow, the  equality we want to prove certainly holds if one shows that 
$\widetilde{E_G}\circ\s_\om^t=\s_\om^t\circ\widetilde{E_G}$ for all $t$. This equality is a matter of easy computation.
Indeed, for $T\in\mathcal{R}_\om$ and $x\in\ch_\om$, one has
\begin{align*}
\widetilde{E_G}(\s_\om^t(T))x=&\int_G \widetilde{\a_g}(\s_\om^t(T))x\,{\rm d} g=
\int_G\s_\om^t(\widetilde{\a_g}(T))x\,{\rm d} g\\
=&\int_G \Delta_\om^{it}\widetilde{\a_g}(T)\Delta_\om^{-it}x\,{\rm d} g=\Delta_\om^{it}\int_G\widetilde{\a_g}(T)\Delta_\om^{-it}x\,{\rm d} g\\
=& \Delta_\om^{it} \widetilde{E_G}(T)\Delta_\om^{-it}x=\s_\om^t(\widetilde{E_G}(T))x\,,
\end{align*}
which proves the sought equality because both $T\in\mathcal{R}_\om$ and $x\in\ch_\om$
are arbitrary.\\
Because $\varphi_\om$ is invariant under the action of the modular flow, we are in a position to apply the so-called Pedersen-Takesaki-Radon-Nikodym theorem (see \emph{e.g.} Theorem 4.10 in \cite{S} or Theorem 2.2.21 in \cite{CS}). Therefore, there must exist a (possibly unbounded) positive operator $H$, defined
on a dense subspace $D(H)\subseteq \ch_\om$, such that $H$ is affiliated with the centralizer of $\widetilde{\om}$ and
$$
\varphi_\om(T)=\widetilde{\om}(HT)\,, \,\, T\in\mathcal{R}_\om\, .
$$
Note that the above equality implies that $\pi_\om(a)\xi_\om$ sits in the domain of $H$ for every $a\in\ga$. In particular,
the GNS vector $\xi_\om$ belongs to the domain of $H$, and \emph{a fortiori} it belongs to the domain of   $H^\frac{1}{2}$ as well. We shall consistently refer to $H$ as the Radon-Nikodym derivative of $\varphi_\om$ with respect to $\om$.\\
We can now get to the conclusion.
Define $W_0$ on the dense subpace $\pi_{\om\circ E_G}(\ga)\xi_{\om\circ E_G}\subset \ch_{\om\circ E_G}$ setting
$$W_0 \pi_{\om\circ E_G}(a)\xi_{\om\circ E_G}:=\pi_\om(a) H^\frac{1}{2}\xi_\om\,,\,\, a\in\ga\,,$$
where $H$ is the  Radon-Nikodym derivative considered above.\\
We next show that $W_0$ is well-defined and isometric. To this end, let us denote by 
$E_n$ the spectral projection of $H^\frac{1}{2}$ corresponding to the interval $[0, n]$. Then
for every $a\in \ga$ we have:
\begin{align*}
&\|\pi_\om(a) H^\frac{1}{2}\xi_\om\|^2= \langle \pi_\om(a) H^\frac{1}{2}\xi_\om, \pi_\om(a) H^\frac{1}{2}\xi_\om\rangle=
 \langle \pi_\om(a^*a) H^\frac{1}{2}\xi_\om, H^\frac{1}{2}\xi_\om\rangle=\\
&\lim_{n\rightarrow\infty} \langle E_n\pi_\om(a^*a) H^\frac{1}{2}\xi_\om, H^\frac{1}{2}\xi_\om\rangle
=\lim_{n\rightarrow\infty} \langle H^\frac{1}{2}E_n\pi_\om(a^*a) H^\frac{1}{2}\xi_\om, \xi_\om\rangle=\\
&\lim_{n\rightarrow\infty} \langle H E_n\pi_\om(a^*a) \xi_\om, \xi_\om\rangle= \langle H\pi_\om(a^*a) \xi_\om, \xi_\om\rangle= \\
&\om\circ E_G(a^*a)=\|\pi_{\om\circ E_G}(a)\xi_{\om\circ E_G}\|^2.
\end{align*}
Denote by $W$ the extension by continuity of $W_0$. The above computations show that $W$ is a linear isometry
of $\ch_{\om\circ E_G}$ into $\ch_\om$. \\
We need to show that $W$ intertwines $\pi_{\om\circ E_G}$  and
$\pi_\om$, that is $W\pi_{\om\circ E_G}(a)=\pi_\om(a) W$ for all $a\in\ga$. It is enough to check the sought equality on the dense subspace $\{\pi_{\om\circ E_G}(b)\xi_{\om\circ E_G}: b\in \ga\}$. But for every $b\in\ga$ one has
\begin{align*}
W\pi_{\om\circ E_G}(a)\pi_{\om\circ E_G}(b)\xi_{\om\circ E_G}&=W\pi_{\om\circ E_G}(ab)\xi_{\om\circ E_G}\\
&=\pi_\om(ab)H^\frac{1}{2}\xi_\om=
\pi_\om(a)\pi_\om(b)H^\frac{1}{2}\xi_\om\\
&=\pi_\om(a) W\pi_{\om\circ E_G}(b)\xi_{\om\circ E_G},
\end{align*}
and we are done. \\
All that is left to do is prove that $W$ is surjective. This is somewhat technical, and the proof is done separately to keep the present
one to a reasonable length.
\end{proof}

The strategy to prove that $W$ is a unitary is to show that $H^\frac{1}{2}\xi_\om$ is still a cyclic vector for $\pi_\om(\ga)$. To do so, we need a couple of preliminary results. 
We recall that a closed densely defined operator $T$ is
said to be non-singular if $T$ is injective, that is, if $\textrm{Ker}(T):=\{x\in D(T): Tx=0\}$ is trivial.  If we assume
that $T$ is a self-adjoint operator, then injectivity of $T$ is clearly the same as asking that $T$ has a dense range.\\

The following general lemma should be a well-known result. Nevertheless, we do provide a detailed proof as it is not easily found in the literature.

\begin{lem}
\label{lem:affiliation}
If $T: D(T)\subset \ch\rightarrow\ch$ is a non-singular self-adjoint operator, then its inverse
$T^{-1}$, defined on its natural domain $D(T^{-1})={\rm Ran}(T)$ is a self-adjoint operator.
Moreover, if $T$ is affiliated to a von Neumann algebra $\mathcal{R}$, then $T^{-1}$ is also
affiliated to $\mathcal{R}$.
\end{lem}

\begin{proof}
The first part of the statement is an easy consequence of the spectral theorem. Precisely, up to unitary
equivalence we can suppose that $\ch= L^2(X, \mu)$ for some finite measure space $(X, \mathfrak{S}, \mu)$, and
that $T$ is the multiplication operator by some real-valued measurable function $h: X\rightarrow \bar{\mathbb{R}}$ which is finite almost everywhere, that is $\mu(\{x\in X: h(x)=\pm \infty\})=0$. The non-singularity of $T$ says that the set $\{x\in X: h(x)=0\}$ is $\mu$-negligible. Therefore, the function $\frac{1}{h}$ is finite almost everywhere w.r.t. $\mu$, which means $T^{-1}$ is still a self-adjoint operator, being in turn a multiplication operator
by a function which is finite almost everywhere.\\
As for the second part of the statement, set $K:=T^{-1}$ and let $U$ be a unitary in $\mathcal{R}'$. We need to show that
$UK\subseteq KU$, that is $x\in D(K)$ implies $Ux\in D(K)$ and $UKx =KUx$. If we define $y:=Kx$, then
$y$ lies in $D(T)$. Since $T$ is affiliated with $\mathcal{R}$ by hypothesis, we have $UT\subseteq TU$. In particular,
$Uy$ is still in $D(T)$ and $UTy =TUy$. From this last equality we see that $Ux=TUy$, so $Ux$ is in 
$D(K)={\rm Ran}(T)$, and $KUx=KTUy=Uy=UKx$.
\end{proof}

\begin{lem}
\label{lem:nonsingular}
The operator $H$ defined as the Radon-Nikodym derivative of $\varphi_\om$ w.r.t $\om$ is non-singular.
\end{lem}

\begin{proof}
It is a consequence of the faithfulness of $\varphi_\om$. Indeed, if $H$ failed to be injective, then the projection
$E$ onto ${\rm Ker}(H)$ would be a non-zero positive element of $\mathcal{R}_\om= \pi_\om(\ga)''$, and we would have
$\varphi_\om(E)=\langle HE\xi_\om,\xi_\om \rangle=0$.
\end{proof}

We are now ready to prove the announced result.
\begin{prop}
\label{prop:isom}
The isometry $W$ is surjective and thus unitary. 
\end{prop}

\begin{proof}
Since the range of $W$ contains the subspace $\{\pi_\om(a) H^\frac{1}{2}\xi_\om: a\in\ga\}$, it is enough
to prove that $H^\frac{1}{2}\xi_\om$ is cyclic for $\pi_\om(\ga)$. By von Neumann's bicommutant theorem, this is the same as 
showing that  $H^\frac{ 1}{2}\xi_\om$ is cyclic for $\mathcal{R}_\om= \pi_\om(\ga)''$. Set $T=H^\frac{1}{2}$ and note 
that $T$ is non-singular by virtue of Lemma \ref{lem:nonsingular}.\\
Let now $x$ be any vector in $\ch_\om$. For every $\varepsilon >0$, by cyclicity of
$\xi_\om$ there exists
$A$ in $\mathcal{R}_\om$ such that $\|A\xi_\om- x \|<\frac{\varepsilon}{2}$.
Thanks to Lemma \ref{lem:affiliation}, $T^{-1}$ is a (positive) self-adjoint operator affiliated with $\mathcal{R}_\om$ as it is in fact affiliated with the centraliser of $\om$. For every $n>0$, let $E_n$ be the spectral
projection of $T^{-1}$ corresponding to the interval $[0, n]$. Clearly, $E_n$ still sits in 
$\mathcal{R}_\om$, which means $B_n:=AE_nT^{-1}$ is a bounded operator that belongs to $\mathcal{R}_\om$.
Because $E_n$ strongly converges to the identity $I$ as $n$ goes to infinity, there must exist $n_0$ such that
$\|E_{n_0}\xi_\om- \xi_\om\|<\frac{\varepsilon}{2\|A\|}$.
But then we have
\begin{align*}
\|B_{n_0}T\xi_\om- x\|\leq&\,\|AE_{n_0}T^{-1}T\xi_\om- A\xi_\om\|+\|A\xi_\om-x\| \leq\\
 &\,\|A\|\|E_{n_0}\xi_\om- \xi_\om\|+\|A\xi_\om-x\| \leq \frac{\varepsilon}{2}+\frac{\varepsilon}{2}=\varepsilon\,,
\end{align*}
which ends the proof as  $\varepsilon$ is arbitrary.
\end{proof}

\section{Examples}
The hypothesis of central support in our main result cannot be dispensed with, as shown by the following class of counterexamples.
\begin{example}
\label{exa:support}
\emph{Let $G$ be any (non-abelian) compact group, and let $U: G\rightarrow \cu(\mathbb{C}^n)$ be a continuous unitary representation, which we will
assume irreducible and with $n\geq 2$. For every $g\in G$, let $\a_g\in{\rm Aut(M_n(\mathbb{C}))}$ be the inner
automorphism implemented by $U(g)$, to wit $\a_g(T):= U_g T U_g^*$, $T\in M_n(\mathbb{C})$.\\
Irreducibility implies that the normalized trace $\varphi$ is the only $G$-invariant state. Indeed, any state $\om$
on $M_n(\mathbb{C})$ is of the form $\om(A)= {\rm Tr} (AT)$, $A\in M_n(\mathbb{C})$, for some positive
$T$ with ${\rm Tr}(T)=1$. $G$-invariance means ${\rm Tr}(AT)={\rm Tr} (U_gAU_g^*T)={\rm Tr} (AU_g^*T U_g) $, hence ${\rm Tr}(A(T-U_g^*TU_g))=0$ for all $A\in M_n(\mathbb{C})$, $g\in G$. By the faithfulness of the trace, we find
$T=U_gTU_g^*$ for all $g\in G$, which means $T=\lambda I$ as the representation is irreducible. Finally, the condition
${\rm Tr}(T)=1$ forces $\lambda$ to be equal to $\frac{1}{n}$, hence $\om$ is the normalized trace on $M_n(\mathbb{C})$.\\
On the other hand, for any state $\om$ on $M_n(\mathbb{C})$ one certainly has $\pi_\om\cong \pi_{\om\circ\a_g}$ for all $g\in G$, merely because the automorphisms $\a_g$ are  inner. This means that any state is automatically 
quasi-invariant. If now $\om_x$ is any vector state, that is $\om_x(A):= \langle Ax, x \rangle $, $A\in M_n(\mathbb{C})$, for some $x\in \mathbb{C}^n$, then $\pi_{\om_x}$ is irreducible. However, $\pi_\varphi$ is not irreducible, being the direct sum of $n$ copies of the identical representation of $\mathbb{M}_n(\mathbb{C})$. This shows that the conclusion of Theorem \ref{thm:main} fails to hold. The reason is that $\om_x$ does not have central support.}
\end{example}

\section*{Acknowledgments}
\noindent
The author acknowledges the support of the Italian INDAM-GNAMPA, the Italian PNRR MUR project PE0000023-NQSTI, CUP H93C22000670006 and Progetto ERC SEEDS UNIBA ``$C^*$-algebras and von Neumann algebras in Quantum Probability'', CUP H93C23000710001.

\end{document}